\documentclass{amsart}
\usepackage{amsmath}
\usepackage{graphicx}
\usepackage{latexsym}

\newtheorem{thm}{Theorem}

\newtheorem{prop}[thm]{Proposition}

\newtheorem{theorem}[thm]{Theorem}
\newtheorem{lemma}[thm]{Lemma}
\newtheorem{corollary}[thm]{Corollary}
\newtheorem{proposition}[thm]{Proposition}

\theoremstyle{definition}
\newtheorem{defn}[thm]{Definition}

\newtheorem{rmk}[thm]{Remark}

\newtheorem{remark}[thm]{Remark}

\newcommand{\Z}{\mathbb{Z}}

\def \x {\times}
\def \eu{{\text{e}}}
\def \sign{{\text{sign}}}

\def\cl{{\rm cl}}
\def\scl{{\rm scl}}

\begin{document}

\title[Sections of surface bundles and Lefschetz fibrations]
{Sections of surface bundles \\ and Lefschetz fibrations}

\author[R. \.{I}. Baykur]{R. \.{I}nan\c{c} Baykur}
\address{Max Planck Institut f\"ur Mathematik, Bonn, Germany \newline
\indent Department of Mathematics, Brandeis University, Waltham MA, USA}
\email{baykur@mpim-bonn.mpg.de, baykur@brandeis.edu}

\author[M. Korkmaz]{Mustafa Korkmaz}
\address{Department of Mathematics, Middle East Technical University, Ankara, Turkey}
\email{korkmaz@metu.edu.tr}

\author[N. Monden]{Naoyuki Monden}
\address{Department of Mathematics, Osaka University, Osaka, Japan}
\email{n-monden@cr.math.sci.osaka-u.ac.jp}

\begin{abstract}
We investigate the possible self-intersection numbers for sections of surface bundles and Lefschetz fibrations over surfaces. When the fiber genus $g$ and the base genus $h$ are positive, we prove that the adjunction bound $2h-2$ is the only universal bound on the self-intersection number of a section of any such genus $g$ bundle and fibration. As a side result, in the mapping class group of a surface with boundary, we calculate the precise value of the commutator lengths of all powers of a Dehn twist about a boundary component, concluding that the stable commutator length of such a Dehn twist is $1/2$. We furthermore prove that there is no upper bound on the number of critical points of genus--$g$ Lefschetz fibrations over surfaces with positive genera admitting sections of maximal self-intersection, for $g \geq 2$.
\end{abstract}

\maketitle

\setcounter{secnumdepth}{2}
\setcounter{section}{0}


\section{Introduction}

Surface bundles over surfaces, and more recently, Lefschetz fibrations, have constituted a rich source of examples of smooth, symplectic, and complex manifolds. The current article explores the existence and the diversity of surface bundles and Lefschetz fibrations which admit sections of maximal possible self-intersection.

Fixing the fiber and the base genera of a surface bundle over a surface, or a Lefschetz fibration, the first question we will tackle is the following: What are the constraints on self-intersection numbers of sections of \emph{all} such maps? The fundamental obstruction hinges on the fact that the total spaces of these maps are symplectic $4$-manifolds and comes from an application of the adjunction inequality for Seiberg-Witten invariants. (Proposition~\ref{adjunctionbounds}.) Namely, for a given genus--$g$ Lefschetz fibration over a surface $\Sigma_h$ of genus $h$ with $g, h \geq 1$, the self-intersection of a section $S$ of this fibration can be at most $2h-2$. For a surface bundle, the self-intersection number of $S$ is also bounded below by $2-2h$. (In this article, whenever we talk about a Lefschetz fibration, we will assume the presence of critical points, so as to make a clear distinction.) We refer to these bounds as \emph{adjunction bounds}, and to sections attaining maximal self-intersection numbers allowed by the adjunction bounds as \emph{maximal sections}.

The natural question to ask, therefore, is whether or not the adjunction bounds are the only universal constraints on self-intersection number of sections for surface bundles or Lefschetz fibrations with fixed fiber and base genera. We take up this problem in Section \ref{Maximal} for surface bundles over surfaces. We prove that the adjunction bounds are the \emph{only} global bounds on self-intersections of sections for all $g \geq 2$ and $h \geq 1$. That is, we prove that
\begin{theorem}
Let $h\geq 1$. For every $g\geq 2$ and $2-2h\leq k \leq 2h-2$, there exists a $\Sigma_g$--bundle over $\Sigma_h$, admitting a section of self-intersection $k$, and in particular, a maximal section.
\end{theorem}
\noindent We, moreover, construct $\Sigma_g$--bundles over $\Sigma_h$ which admit disjoint sections that attain all the possible self-intersection numbers between $2-2h$ and $2h-2$ for any $g\geq 8h-8$ and $h \geq 1$. (Theorem \ref{maintheorem1}.)  These constructions make use of mapping class group relations we present in the same section.

For the remaining values of $g$ and $h$, we have a complete treatment, which demonstrates a contrast with our
results mentioned above: When $g$ or $h$ is zero, the conclusions are classical: For $g=0$, we have ruled surfaces, that is, $S^2$--bundles over $\Sigma_h$, which can admit sections with arbitrary self-intersections, and for $g \geq 1$ and $h=0$, there is indeed a unique $\Sigma_g$--bundle over $S^2$ given by projection onto the second factor of $\Sigma_g \x S^2$. On the other hand, in the case of $g=1$ and $h \geq 1$, we observe that another extreme situation arises: The sections of these elliptic surface bundles \emph{always} have self-intersection zero (Proposition \ref{lowgeneracases}), even though the adjunction bounds hand us a larger range from $2-2h$ to $2h-2$ for $h \geq 1$.

A side result of particular interest is given in Theorem~\ref{thm:scl}, where we compute the commutator length of powers of the Dehn twist about a boundary component, and in turn, the stable commutator length of such a Dehn twist. Namely we show:

\begin{theorem}
Let $g\geq 2$, $n\geq 1$ and let $\Sigma$ be a compact connected oriented surface of genus $g$ with boundary. If $\delta$ is one of the boundary components of $\Sigma$, then \linebreak the commutator length of $t_{\delta}^n$ is \, $\lfloor (|n|+3)/2 \rfloor$ \, and the stable commutator length of $t_{\delta}$ is $1/2$.
\end{theorem}
\noindent What follows is a stable calculation which is in great contrast with a result of Kotschick, who showed that the stable commutator length function vanishes when one considers the stable mapping class group of compact surfaces with \textit{one} boundary component~\cite{kot}. It follows from our result that the stable commutator length function does not vanish when one works with the stable mapping class group of compact surfaces with \textit{more than one} boundary component instead. These results are given in Subsection~\ref{sclresults}. Some computations and estimates on the commutator length and the stable commutator length of a (multi)twist were obtained in~\cite{KorkOzbag, EK, Kork, kot2, BraunKot}.

Lastly, we address the following question: For fixed fiber and base genera, is there an upper bound on the number of singular fibers of a relatively minimal Lefschetz fibration admitting a section of maximal self-intersection? As the maximal self-intersection number for a Lefschetz fibration (again, with critical points) over the $2$-sphere is $-1$ \cite{Smith, Stipsicz}, when the base genus $h=0$ this question is due to Ivan Smith (presented by Denis Auroux in \cite{Au}). For fiber genus $g \geq 2$ and base genus $h \geq 1$, we show that the answer to this question is negative by proving the following:
\begin{theorem}
Let $g\geq 2$ and $h\geq 1$ be fixed integers. For any positive integer $M$, there exists a relatively minimal genus--$g$ Lefschetz fibration over a surface of genus $h$ admitting a maximal section such that the number of critical points is greater than $M$.
\end{theorem}
\noindent These results are collected in Section~\ref{lfresults}.

\newpage
\section{Preliminaries}

\subsection{Lefschetz fibrations and monodromy representations.} \

In this article, all manifolds are assumed to be smooth and oriented, and all maps are assumed to be smooth. We denote by $\Sigma_{g,r}^s$ a compact oriented surface of genus $g$ with $s$ boundary components and $r$ marked points in the interior. The \emph{mapping class group\,}, $\Gamma_{g,r}^s$, of  $\Sigma_{g,r}^s$ is the group of isotopy classes of orientation-preserving self-diffeomorphisms of $\Sigma_{g,r}^s$ fixing $r$ marked points and the points on the boundary. The isotopies of $\Sigma_{g,r}^s$ are assumed to fix the marked points and the points on the boundary. For simplicity, we write $\Sigma_{g,r} = \Sigma_{g,r}^0$, $\Sigma_{g}^s = \Sigma_{g,0}^s$ and $\Sigma_{g} = \Sigma_{g,0}^0$. We also use the similar simplified notation for the corresponding mapping class groups.

We start with reviewing some basic definitions and properties of Lefschetz fibrations and surface bundles over surfaces.

Let $X$ and $\Sigma$ be compact connected manifolds of dimensions four and two, respectively. A \emph{Lefschetz fibration} is a map $f\colon X\to\Sigma$ such that $f^{-1} (\partial \Sigma)=\partial X$, the set $C =\{p_1,p_2,\ldots,p_k \}$ of critical points of $f$ lies in the interior of $X$, and that around each $p_i$ and $f(p_i)$ there are orientation-preserving charts making $f$ conform to the complex model $f(z_1, z_2)=z_1 z_2$. The genus $g$ of a regular fiber $F \cong \Sigma_g$ of $f$ is called the \emph{genus of the fibration}. We will assume that the critical points lie in different fibers, called \emph{singular fibers}, which can be achieved after a small perturbation of any given Lefschetz fibration. When there are no critical points, $f: X \to \Sigma$ is nothing but a surface bundle over a surface, so $f$ always restricts to a surface bundle over $\Sigma \setminus f(C)$ on $X \setminus f^{-1}(f(C))$ and, in particular, over $\partial \Sigma$ on $\partial X$. Below, whenever we talk about a Lefschetz fibration, we will assume that the critical locus is non-empty, so as to make a clear distinction from surface bundles over surfaces.

A singular fiber is called \emph{reducible} if the complement of the critical point in the fiber is disconnected, and is called \emph{irreducible} otherwise. Lefschetz singularities locally correspond to $2$-handle attachments to $D^2 \x F$ with framing $-1$ with respect to the fiber framing, where the attaching circles of these $2$-handles are embedded curves in a regular fiber $F$ and are called \emph{vanishing cycles}. With this in mind, a reducible (resp. irreducible) singular fiber is given by a separating (resp. non-separating) vanishing cycle on $F$. So a reducible fiber consists of two surfaces of self-intersection $-1$ intersecting each other at the critical point. If one of these two surfaces is a $2$-sphere, that is, if the vanishing cycle is null-homotopic on $F$, then one gets a new Lefschetz fibration by blowing-down this sphere without altering the rest of the fibration, and vice versa. Therefore, we may consider only \emph{relatively minimal} Lefschetz fibrations, i.e. fibrations which do not contain any sphere of self-intersection $-1$ in its fibers.

Lefschetz fibrations can be described combinatorially by means of their monodromy. For a smooth surface bundle $f\colon E\to \Sigma$ with fibers diffeomorphic to $F$, the {\em monodromy representation} of $f$ is defined to be the map $\Psi \colon \pi _1(\Sigma)\to \Gamma _g$  relative to a fixed identification $\varphi$ of $F$ with the fiber over the base point of $\Sigma$: For each loop $\gamma \colon I\to \Sigma$ the bundle $f_\gamma \colon \gamma^* (E)\to I$ is canonically trivial, inducing a diffeomorphism $f_\gamma^{-1}(0)\to f_\gamma^{-1}(1)$ up to isotopy. Using $\varphi$ to identify $f_\gamma^{-1}(0)$ and $f_\gamma^{-1}(1)$ with $F$, we get the element $\Psi (\gamma)\in \Gamma _g$. Changing the identification $\varphi$ changes $\Psi$ by a conjugation with an element of $\Gamma_g$. Here the map $\Psi \colon \pi _1(\Sigma)\to \Gamma _g$ is an antihomomorphism rather than a homomorphism, because for the multiplication in the mapping class group we use the functional notation, i.e. for $f_1,f_2\in\Gamma_g$, the product $ f_1f_2$ means that we first apply $f_2$ and then $f_1$.

For a relatively minimal, genus--$g$ Lefschetz fibration $f \colon X\to\Sigma$ with a regular fiber $F$, we define the \emph{monodromy representation} (or simply \emph{monodromy}) to be the \emph{monodromy factorization} of the fiber bundle $X \setminus f^{-1}(Q)\to \Sigma \setminus Q$, where $Q=f(C)$ is the set of critical values. For $f \colon X\to \Sigma$ as above, the monodromy representation $\Psi \colon \pi_1 (\Sigma \setminus Q) \to \Gamma _g$ determines $f$ up to isomorphism, except in the cases of sphere and torus bundles over closed surfaces. This is due to the fact that for $g\geq 2$ the space of self-diffeomorphisms of $F$ isotopic to the identity is contractible.


It turns out that the monodromy of a Lefschetz fibration $f\colon X\to D^2$ over the disk with a single
critical point is a right Dehn twist\index{Dehn twist} along the vanishing cycle corresponding to the singular
fiber. Therefore, the monodromy of a Lefschetz fibration $f\colon X \to \Sigma _h$ is given by a factorization of the identity element $1\in \Gamma _g$ as
\begin{equation} \label{monodromypresentation}
1=\prod _{i=1}^n t_{v_i}\prod _{j=1}^h [\alpha_j, \beta_j] \, ,
\end{equation}
where $v_i$ are the vanishing cycles of the singular fibers and $t_{v_i}$ is the \textit{right} (or \textit{positive})  Dehn twist about $v_i$. This factorization of the identity is called the \emph{monodromy factorization}. In particular, for $F = \Sigma_g$, a product $\prod_{i=1}^h [a_i, b_i]$ of $h$ commutators in $\Gamma_g$ gives a $\Sigma_g$--bundle over the surface $\Sigma_h^1$
of genus $h$ with one boundary component. The mapping classes $a_i$ and $b_i$ specify the monodromies along
a free generating system $\langle \alpha_1, \beta_1, \ldots , \alpha_h, \beta_h \rangle$ of $\pi_1 (\Sigma_h^1)$
such that $\prod_{i=1}^h [\alpha_i, \beta_i]$ is parallel to the boundary component of $\Sigma_h^1$. If $\prod_{i=1}^h [\alpha_i, \beta_i]=1$ in $\Gamma_g$, we get a $\Sigma_g$--bundle $X\to \Sigma_h$. The bundle is uniquely determined by this factorization of the identity once $g\geq 2$.

Conversely, a product $\prod_{i=1}^k t_i\in \Gamma_g$ with $t_i$ right Dehn twists provides a genus--$g$ Lefschetz fibration $X\to D^2$ over the disk with fiber $F \cong \Sigma_g$. So if $\prod_{i=1}^kt_i=1$ in the mapping class group $\Gamma_g$, then the fibration closes up to a fibration over the sphere $S^2$ and the closed--up manifold is uniquely determined by the word $\prod_{i=1}^k t_i$ once $g\geq 2$.  By combining the above two constructions, a word
\[
w=\prod_{i=1}^{k'} t_i \prod_{j=1}^h [\alpha_i, \beta_i] \]
gives a Lefschetz fibration over $\Sigma_h \setminus D^2$, and if $w=1$ in $\Gamma_g$ we get a Lefschetz fibration $X\to \Sigma_h$.

For a Lefschetz fibration or a surface bundle $f: X\to \Sigma$, a map $\sigma  \colon \Sigma  \to X$ 
is called a \emph{section} if $f \circ \sigma = id_{\Sigma}$. Suppose that a fibration $f\colon X\to \Sigma$ 
admits a section $\sigma$. Set $S = \sigma(\Sigma) \subset X$. We will also say that $S$
is a \emph{section} of $f$. This section $S$ provides a lift of the 
representation $\Psi$ from $\pi_1 (\Sigma \setminus Q)$ to the mapping class group $\Gamma _{g,1}$. 
One can then fix a disk neighborhood of this section preserved under the monodromy, and get a lift to $\Gamma_g^1$.

Conversely, every such representation with a lift determines a fibration with a section: Gluing a disk with a marked point to a surface with one boundary component along the boundary and by extending self-diffeomorphisms of the surface by the identity on the disk, we obtain a surjective homomorphism $p\colon \Gamma ^1_g\to \Gamma _{g,1}$. It is well known that $\ker p$ is isomorphic to $\Z$, generated by the right Dehn twist $t_\delta$ along a simple closed curve $\delta$ parallel to the boundary. If the factorization
 \[
 1=\prod _i t_{v_i} \prod _j [\alpha_j, \beta_j]
 \]
lifts from $\Gamma _g $ to $\Gamma _{g,1} $ as a factorization of $1$ in the latter group in the similar form, then the corresponding fibration has a section. Moreover, if we lift this product to $\Gamma _g^1$ we get
  \[
  t_\delta ^m=\prod _i t_{v'_i} \prod _j [\alpha'_j, \beta'_j]
   \]
for some $m$. Here, $t_{v'_i}$ is a Dehn twist mapped to $t_{v_i}$ under $\Gamma _g^1 \to \Gamma _g$. Similarly, $\alpha_j'$ and $\beta'_j$ are mapped to $\alpha_j$ and $\beta_j$, respectively. An elementary observation is that the power $m$ of $t_{\delta}$ in the above factorization in $\Gamma _g^1$ is the negative of the self-intersection number of the section $S$ that we obtain. (See for instance \cite{Smith}.)


\subsection{Background results} \

There are three basic operations we are going to use to construct new surface bundles and Lefschetz fibrations from given ones:

\noindent \rm{(1)}
Let $f: X \to \Sigma_h$ be a Lefschetz fibration with regular fiber $\Sigma_g$. If $\lambda$ is an \textit{orientation-preserving} self-diffeomorphism of $\Sigma_h$, then $\lambda \circ f : X \to \Sigma_h$ is also a Lefschetz fibration with regular fiber $\Sigma_g$. If $f: X \to \Sigma_h$ is a $\Sigma_g$--bundle, then one can take the diffeomorhism $\lambda$ above to be orientation-reversing as well to get a new $\Sigma_g$--bundle $\lambda \circ f : X \to \Sigma_h$. \\

\noindent \rm{(2)}
For $i=1,2$, let $f_i: X_i \to \Sigma_{h_i}$ be two genus--$g$ Lefschetz fibrations. We can then remove a fibred neighborhood of a regular fiber $\Sigma_g$ from each fibration and glue the resulting $4$-manifolds along their boundaries using a fiber--preserving and orientation-reversing diffeomorphism $\phi$ of $S^1 \x \Sigma_g$ to get a new oriented $4$-manifold $X$. The result is a new genus-$g$ Lefschetz fibration
$f = f_1 \#_{\phi} f_2: X  \to \Sigma_{h_1+h_2}$, called the \emph{fiber sum} of $f_1$ and $f_2$. Moreover, if $S_i$ is a section of $f_i$ with self-intersection $k_i$ for $i=1,2$, then one can perform this fiber sum
operation so that there is a section $S_1 \# S_2$, restricting to $S_i$ on each fiber sum component, of the fibration $f = f_1 \#_{\phi} f_2: X  \to \Sigma_{h_1+h_2}$ with self-intersection $k= k_1 + k_2$. In this case, we say that the pair $(f,S)$ is the \emph{fiber sum} of $(f_1, S_1)$ and $(f_2, S_2)$. \\

\noindent \rm{(3)}
Let $f_i: X_i \to \Sigma_{h}$ be a genus--$g_i$ Lefschetz fibration with the regular fiber
$\Sigma_{g_i}$ with a self-intersection zero section $S_i$ for each $i=1,2$. We can then remove a $D^2$ fibered
neighborhood of each $S_i$ in $X_i$ and glue the resulting $4$-manifolds along their boundaries using a fiber
preserving orientation-reversing diffeomorphism $\phi$ of $S^1 \x \Sigma_h$. The result is a new Lefschetz
fibration $f = f_1 \#_{\phi} f_2: X \to \Sigma_h$ with the regular fiber $\Sigma_{g_1+g_2}$.


The constructions \rm{(2)} and \rm{(3)} above are both instances of the generalized fiber sum construction, although
they appear to be ``orthogonal'' to each other in nature. When the fibers are homologically essential (in particular when the fiber genus $g \neq 1$), these Lefschetz fibrations can be equipped with symplectic forms that make the fibers \textit{and} any prescribed finite collection of disjoint sections symplectic, allowing us to perform this generalized fiber sum construction symplectically and handing us a new symplectic Lefschetz fibration at the end. It is not hard to express all these operations in terms of factorizations and lifts in appropriate mapping class groups.


Next proposition prescribes the \emph{adjunction bounds} for self-intersection numbers
of sections of surface bundles and of Lefschetz fibrations:

\begin{prop} [Adjunction bounds] \label{adjunctionbounds}
Let $f: X \to \Sigma_h$ be a surface bundle or \linebreak a Lefschetz fibration with regular fiber $\Sigma_g$, and $g, h \geq 1$. Suppose that $f$ admits a section $S$. Then the self-intersection of $S$ satisfies $[S]^2 \leq 2h-2$. If $f: X \to \Sigma_h$ is a $\Sigma_g$ bundle, then this bound improves to $| \, [S]^2 \, | \leq 2h-2 $.
\end{prop}

\begin{proof}
If $f$ is a surface bundle over a surface with positive fiber and base genera, we easily deduce from the homotopy exact sequence of a fibration that the total space $X$ is acyclic, and in particular minimal. On the other hand, if $f$ is a Lefschetz fibration which is not relatively minimal, then we can always blow-down the exceptional spheres on the fibers and pass to a relatively minimal genus $g$ Lefschetz fibration over a genus $h$ surface. Since a section intersects each fiber positively at one point, it either intersects each one of these exceptional spheres positively once or misses it. We therefore obtain a relatively minimal Lefschetz fibration with a section whose self-intersection number is greater than or equal to the self-intersection number of the original section. Thus, it suffices to prove the proposition for a relatively minimal Lefschetz fibration $f$, which we will assume from now on. Then it follows from \cite{Stipsicz0} that the total space of this fibration is minimal. To sum up: in all cases we can assume that we are working with a minimal $X$.

Since the fibration admits a section, the fibers are homologically essential, allowing us to equip $X$ with
a symplectic form using Gompf-Thurston construction. We therefore have a minimal symplectic $4$-manifold $X$ in hand.
In particular, $b^+(X) \geq 1$. 

\noindent \textit{The case $b^+(X) > 1$}:
The proof in this case will follow at once
from an application of the adjunction inequality for Seiberg-Witten basic classes.
As shown by Taubes, the canonical class $K$ of any symplectic form we choose on $X$ is a Seiberg-Witten basic class.
Applying the adjunction inequality we get
\[ -\chi(S) \geq [S]^2 + | K \cdot S | \, . \]
It follows that $2h-2 \geq [S]^2$, proving the first claim of the proposition in this case.

\noindent \textit{The case $b^+(X) = 1$}:
First observe that if $F \cong \Sigma_g$ is a regular fiber, then for any $r > \frac{1}{2}|[S]^2|$ the class $r[F]+ [S]$
has positive square, whereas $r[F]-[S]$ has negative square. So $b^-(X)\geq 1$. Moreover, from the exact sequence
\[
\pi_1(F) \rightarrow \pi_1(X) \stackrel{f_{*}}{\rightarrow} \pi_1(\Sigma_h) \rightarrow 1 \, ,
\]
we get $b_1(X) \geq b_1(\Sigma_h) = 2h$. (See for instance ~\cite{GS}.)

Second, we show that $X$ is not ruled. From $b_1(X) \geq 2h$, it is fairly easy to see that $X$ can only be a ruled surface over $\Sigma_n$ with $n \geq h$. In this case, from the Euler characteristic calculation $4(g-1)(h-1) + m =  4-4n$, where $m$ is the number of Lefschetz critical points, we get
\[
0 \leq \frac{m}{4} = -gh +g+h -n \leq g(1-h)\leq 0 \, .
\]
Hence, $m=0$ (and $h=n=1$), i.e. $f$ should be a surface bundle over a surface. Since the base and the fiber genera are positive, this in turn implies that $\pi_2(X)=0$, so $X$ can not be ruled.

Since $X$ is a minimal symplectic $4$-manifold with $b^+(X)=1$ and since it is not (irrational) ruled,
as shown by Liu in \cite{Liu}, we have the inequality
\begin{equation*}
0 \leq c_1^2(X)= 2 \, \eu(X) + 3 \, \sign(X) \, ,
\end{equation*}
where $\eu(X)$ and $\sign(X)$ are the Euler characteristic and the signature of $X$, respectively.
From this equation we derive that
\[ 0\leq 4 - 4 b_1(X) +5 b^+(X) - b^-(X) = 9 - 4 b_1(X)  - b^-(X). \]
Therefore, we get
\[ 0 \leq 8h + b^-(X) \leq 4 b_1(X)  + b^-(X)\leq 9 \, , \]
giving us $b_1(X)=2, h=1$ and $b^-(X)=1$. In particular, the Euler characteristic of $X$ is
$ \eu(X)= 0$. Calculating the Euler characteristic from the handle decomposition of $X$ we have
\[ 4(g-1)(h-1) + m = 0 \, ,\]
where $m$ is the number of singular fibers of the Lefschetz fibration $f$ (considering the surface bundle as
a Lefschetz fibration with no singular fiber). We conclude that $m=0$, that is, we have an honest
$\Sigma_g$--bundle over the torus $T^2$.

Let $[S]^2=s$. This hands us a relation $t_\delta ^{-s}=[\alpha', \beta']$ in
the mapping class group $\Gamma_g^1$ of a surface $\Sigma_g^1$ with one boundary component,
which is a lift of the monodromy factorization $[\alpha, \beta]=1$ in $\Gamma_g$ of the surface bundle $f$.
Here $\delta$ is the boundary curve $\Sigma_g^1$. Gluing a torus with one boundary component to $\Sigma_g^1$ along $\delta$ gives
an embedding $\Sigma_g^1 \hookrightarrow \Sigma_{g+1}$, which in turn
induces an injection $\Gamma_g^1 \hookrightarrow \Gamma_{g+1}$. Thus we get the relation
$t_{\delta'}^{-s}= [\alpha'', \beta'']$ in $\Gamma_{g+1}$, where $\delta'$ is a homotopically
nontrivial separating simple closed curve on $\Sigma_{g+1}$. However, it was shown by Endo and Kotschick in \cite{EK} that no nontrivial power of a Dehn twist along a separating simple closed curve is a commutator. Therefore, $[S]^2=0$, concluding the proof of the first claim.

If $f$ is an honest $\Sigma_g$--bundle over $\Sigma_h$ with a section $S$ of self-intersection $k$, then we have
\[ [\alpha_1, \beta_1][\alpha_2, \beta_2]\cdots [\alpha_h, \beta_h] = t_{\delta}^{-k} \]
in $\Gamma_g^1$, where $\delta$ is the boundary component.
By inverting this equation we obtain the relation
\[ [\beta_h,\alpha_h] \cdots [ \beta_2,\alpha_2] [\beta_1,  \alpha_1] = t_{\delta}^k \]
which gives another $\Sigma_g$--bundle over $\Sigma_h$, namely the \textit{reflection} of $f$. This bundle
have a section of self-intersection $-k$. By the result above $ -k \leq 2h-2$, thus $ |[S]^2|\leq 2h-2 $, concluding the proof. 
\end{proof}

While circulating an earlier version of this article, we found out that such an adjunction bound in
the case of \textit{surface bundles} of fiber and base genera $g, h \geq 2$ was independently obtained
by Bowden \cite{Bow}, where the author studies multisections of surface bundles.

Note that when we have a Lefschetz fibration, there is indeed no lower bound on the self-intersection
of a section. This can for instance be seen by taking fiber sums with appropriate Lefschetz fibrations
over the $2$-sphere with negative square sections while patching the sections.

When the base genus is $h=0$, a maximally self-intersecting section for a Lefschetz fibration (again with
non-empty critical locus) has self-intersection $-1$, as shown in \cite{Smith, Stipsicz}. In the light of
above proposition, we make the following definition:

\begin{defn}
Let $f: X \to \Sigma_h$ be a Lefschetz fibration with the regular fiber $\Sigma_g$ and let $S$ be a section
of this fibration. Then $S$ is called a \emph{maximally self-intersecting section}, or a \emph{maximal section}
in short, if $[S]^2 = 2h-2$ when $h \geq 1$, and $[S]^2= -1$ when $h =0$.
\end{defn}

Let us now define:

\begin{defn}[Fiber sum indecomposability for pairs] \label{fibersumindecomposable}
Let $f$ be a surface bundle or a Lefschetz fibration over $\Sigma_{h}$ with the regular fiber $\Sigma_{g}$ and let $S$ be a section of this fibration. If $(f,S)$ can be expressed as the fiber sum of $(f_1, S_1)$ and $(f_2, S_2)$ for some surface bundles or Lefschetz fibrations $f_i: X_i \to \Sigma_{g_i}$ with sections $S_i$ such that neither one of $(f_i, S_i)$ is the trivial $\Sigma_g$--bundle over $S^2$, then $(f,S)$ is said to be
\emph{fiber sum decomposable}. The pair $(f, S)$ is called \emph{fiber sum indecomposable} otherwise. Fiber sum indecomposability for a fibration $f$ alone (without any mentioning of a section) is defined similarly.
\end{defn}

\begin{proposition} \label{fibersumindecomposablity}
Let $g, h \geq 1$. If $f: X \to \Sigma_h$ is a surface bundle or a Lefschetz fibration with the regular fiber $\Sigma_g$ admitting a maximal section $S$, then the pair $(f, S)$ is fiber sum indecomposable.
\end{proposition}

\begin{proof}
Let $(f,S)$ be the fiber sum of two genus--$g$ fibrations $(f_1,S_1)$ and $(f_2, S_2)$ over surfaces of genera $h_1$ and $h_2$, respectively. Assume that $0 < h_1 \leq h_2 <h$. We have $[S_1]^2 + [S_2]^2 = [S]^2 = 2h-2$ and $h_1 + h_2 =h$. However, by Proposition \ref{adjunctionbounds}, $[S_i]^2 \leq 2h_i -2$ for $i=1,2$, yielding a contradiction. So now assume that $h_1 =0$ and $h_2 =h$. Once again by Proposition \ref{adjunctionbounds}, $[S_2]^2 \leq 2h-2$. It follows that $[S_1]^2 =0$, so $f_1$ is the trivial $\Sigma_g$--bundle \cite{Smith, Stipsicz}. Thus $(f,S)$ is fiber sum indecomposable.
\end{proof}

\begin{rmk} \label{fibersumpositivebasegenus}
If the section $S$ has self-intersection $2h-3$, that is one less than the maximum possible self-intersection, then the first part of the above proof still works and shows that \emph{the pair} $(f,S)$ cannot be decomposed as fiber sum of two Lefschetz fibrations over surfaces. However, we shall note that this does not exclude the possibility of the fibration $f$ decomposing as a non-trivial fiber sum; our observation here points out when a fibration \emph{and} a section are \emph{not} obtained from fiber summing non-trivial fibrations while pacing their sections.
\end{rmk}




Lastly, the following proposition will be handy when addressing the main questions of our paper in the case of fiber genus $g=1$:

\begin{proposition} \label{propg=1}
Let $\Sigma_1^1$ be a torus with one boundary component $\delta$ and let $n$ be an integer. In the mapping class group $\Gamma_{1}^1$ of $\Sigma_{1}^1$,
\begin{itemize}
  \item if $n<0$ then $t_\delta^n$ cannot be written as a product of commutators and right nonseparating Dehn twists,
  \item if $n\geq 0$ and if $t_\delta^n$ is a product of commutators and $m$ right nonseparating Dehn twists, then $m=12n$.
\end{itemize}
\end{proposition}

\begin{proof}
The mapping class group $\Gamma_{1}^1$ is isomorphic to the braid group on three strands, and is generated by $t_a,t_b$ for any two nonseparating simple closed curves intersecting at one point. The first integral homology group $H_1(\Gamma_1^1)$ is isomorphic to $\Z$ (c.f~\cite{k02tjm} Theorem~5.1). Let us identify $H_1(\Gamma_1^1)$ with $\Z$ so that the Dehn twist about a nonseparating curve represents $1 \in \Z$.

Note that we have the equality $t_\delta=(t_at_b)^6$, so that $t_\delta$ represents $12$ in $H_1(\Gamma_1^1) \cong \Z$. It follows that if $t_\delta^n$ is written as a product of $h$ commutators and $m$ right nonseparating Dehn twists, then $12n=m$. The proposition follows from this.
\end{proof}

\section{Commutator lengths and surface bundles with maximal sections} \label{Maximal}

If $x$ and $y$ are elements of a group $G$, then we write $[x,y]=xyx^{-1}y^{-1}$ and call it a \emph{commutator}. The
purpose of this section is to write the powers of the Dehn twist about a boundary component on a surface
as a product of least possible number of commutators. More precisely, we will compute the commutator length and the stable
commutator length of such a Dehn twist. We then apply it to surface bundles in this section and to Lefschetz fibrations
in the next section.

The following lemma will be useful for us. The case $k=1$ was used in~\cite{KorkOzbag}.

\begin{lemma} \label{lem:comm}
Let $\Sigma$ be a compact connected oriented surface. Let $a,b,c,d$ be simple closed curves on
$S$ such that there exists a diffeomorphism $f$ mapping $(a,b)$ to $(d,c)$. Then
$t_a^kt_b^{-k} t_c^kt_d^{-k}$ is a commutator.
\end{lemma}
\begin{proof}
$t_a^kt_b^{-k} t_c^kt_d^{-k}=t_a^kt_b^{-k} t_{f(b)}^kt_{f(a)}^{-k}= t_a^kt_b^{-k} f (t_{b}^kt_{a}^{-k}) f^{-1}=[t_a^kt_b^{-k},f]$.
\end{proof}

\subsection{Commutator length of the Dehn twist about a boundary component} \

\subsubsection{Lantern relation and its generalization}
Let $D'$ be a disk with boundary $\delta'$ from which the interior of three disjoint disks
$D_i'$ are removed. Let $a'_1,a'_2,a'_3$ be the resulting boundary components, i.e., $a'_i=\partial D'_i$. Consider the simple closed curves
$x'_1,x'_2,x'_3$ on $D'$ shown in Figure~\ref{lantern-57}(i). Then we have the lantern relation
\begin{eqnarray} \label{eqn:lantern}
t_{\delta'} \, t_{a'_1}t_{a'_2}t_{a'_3}=t_{x'_1}t_{x'_2}t_{x'_3}.
\end{eqnarray}
We note that by the Dehn twist about a boundary component we mean the Dehn twist about a simple closed
curve parallel to that boundary component.

Suppose now that $D$ is a disk with boundary $\delta$ from which the interior of four disjoint disks
$D_i$ are removed. Let $a_1,a_2,a_3,a_4$ be the resulting boundary components, i.e., $a_i=\partial D_i$. Consider the simple closed curves
$x_1,x_2,x_3,x_4,y_1,y_2$ on $D$ shown in Figure~\ref{lantern-57}(ii).

\begin{figure}[hbt]
\begin{center}
\includegraphics[width=13cm]{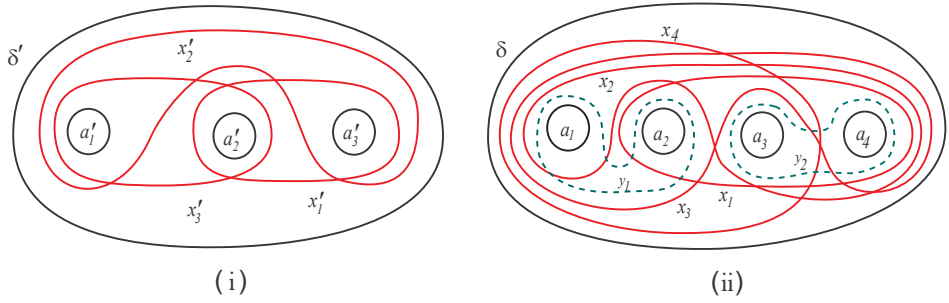}
\caption{The curves of the lantern relation and the generalized lantern relation.}
\label{lantern-57}
\end{center}
\end{figure}

\begin{proposition}
Referring to the simple closed curves on Figure~\ref{lantern-57}(ii),
the following generalization of the lantern relation holds in the mapping class group of the four-holed disk
$D$:
\begin{eqnarray} \label{eqn:alt-delta}
t_\delta^{2} \, t_{a_1}t_{a_2}t_{a_3}t_{a_4}=t_{x_1}t_{x_2}t_{x_3}t_{x_4}.
\end{eqnarray}
\end{proposition}

\begin{proof}
Note that the simple closed curves $\delta, a_1,a_2$ and $y_2$ bound a sphere
with four holes. By the lantern relation we have
\begin{eqnarray} \label{eqn:lantern1}
t_\delta \, t_{a_1}t_{a_2} t_{y_2}= t_{y_1}t_{x_1}t_{x_2}.
\end{eqnarray}
Note also that $\delta, y_1,a_3$ and $a_4$ also bound a sphere
with four holes. Again by the lantern relation we have
\begin{eqnarray} \label{eqn:lantern2}
t_\delta \, t_{y_1}t_{a_3}t_{a_4}= t_{x_3}t_{x_4}t_{y_2}.
\end{eqnarray}

From the relations (\ref{eqn:lantern1}) and (\ref{eqn:lantern2}),
we get
\begin{eqnarray} \label{eqn:lantern3}
t_\delta^2 \, t_{a_1}t_{a_2}t_{a_3}t_{a_4}t_{y_1}t_{y_2}= t_{y_1}t_{x_1}t_{x_2}t_{x_3}t_{x_4}t_{y_2}.
\end{eqnarray}
By canceling $t_{y_1}$ and $t_{y_2}$ from the equation (\ref{eqn:lantern3}) we obtain the desired result,
the equality (\ref{eqn:alt-delta}).
\end{proof}

\bigskip

\begin{remark} [Generalized lantern relation] The above proof can be generalized in a straightforward fashion to obtain the relation:
\begin{eqnarray*}
t_\delta^{n-2} \,t_{a_1}t_{a_2} \cdots t_{a_n} = t_{x_1}t_{x_2}\cdots t_{x_n}
\end{eqnarray*}
on a disk with $n$ boundary components.\footnote{As Burak Ozbagci kindly pointed out to us, this generalized lantern relation is equivalent to a relation obtained by Plamenevskaya and Van Horn-Morris in \cite{PV}.}
\end{remark}

We are now ready to prove the following theorem:

\begin{theorem} \label{thm:ez-d^2n} Let $g\geq 2$, $n\geq 1$ and let $\Sigma$ be a compact connected oriented surface of genus $g$ with boundary. Let $\delta$ be one of the boundary components of $\Sigma$. Then the $n^{th}$ power $t_\delta^{n}$ of the Dehn twist about $\delta$ is a product of  $\lfloor (n+3)/2 \rfloor $  commutators, where $\lfloor (n+3)/2 \rfloor $ is the largest integer not greater than $(n+3)/2$.
\end{theorem}

\begin{proof}
First, we rewrite (\ref{eqn:alt-delta}) as
\begin{eqnarray*}
t_\delta^{2} t_{a_1}t_{a_2}t_{a_3}t_{a_4}t_{x_4}^{-1}= t_{x_1}t_{x_2}t_{x_3}.
\end{eqnarray*}
The Dehn twists on the left--hand side of the above equality commute with each other. So for each positive integer $k$, we obtain
\begin{eqnarray*}
t_\delta^{2k} t_{a_1}^k t_{a_2}^k t_{a_3}^k t_{a_4}^k t_{x_4}^{-k}&=&(t_{x_1}t_{x_2}t_{x_3})^k\\
&=& (t_{x_1}t_{x_2}) (t_{x_1}t_{x_2})^{t_{x_3}}  (t_{x_1}t_{x_2})^{t_{x_3}^2} \cdots (t_{x_1}t_{x_2})^{t_{x_3}^{k-1}}
t_{x_3}^k\\
&=& \left( \prod_{i=1}^{k}  (t_{x_1}t_{x_2})^{t_{x_3}^{i-1}} \right)t_{x_3}^k,
\end{eqnarray*}
or
 \begin{eqnarray}\label{eqn:qqq}
t_\delta^{2k} t_{a_1}^k t_{a_2}^k t_{a_3}^k t_{a_4}^k
 &=&\left( \prod_{i=1}^{k}  (t_{x_1}t_{x_2})^{t_{x_3}^{i-1}} \right)t_{x_3}^k t_{x_4}^k.
\end{eqnarray}
Here, $t_a^f $ denotes the conjugation $ft_af^{-1}$. By multiplying both sides of~\eqref{eqn:qqq}
by $t_{a_1}^{-k} t_{a_2}^{-k} t_{a_3}^{-k} t_{a_4}^{-k}$, we get
 \begin{eqnarray}\label{eqn:delta2nn}
t_\delta^{2k}
 &=&\left( \prod_{i=1}^{k}  (t_{x_1}t_{a_2}^{-1} t_{x_2} t_{a_1}^{-1})^{t_{x_3}^{i-1}} \right)
 (t_{x_3}^k t_{a_4}^{-k} t_{x_4}^k t_{a_3}^{-k} ).
\end{eqnarray}

\begin{figure}[hbt]
 \begin{center}
      \includegraphics[width=11cm]{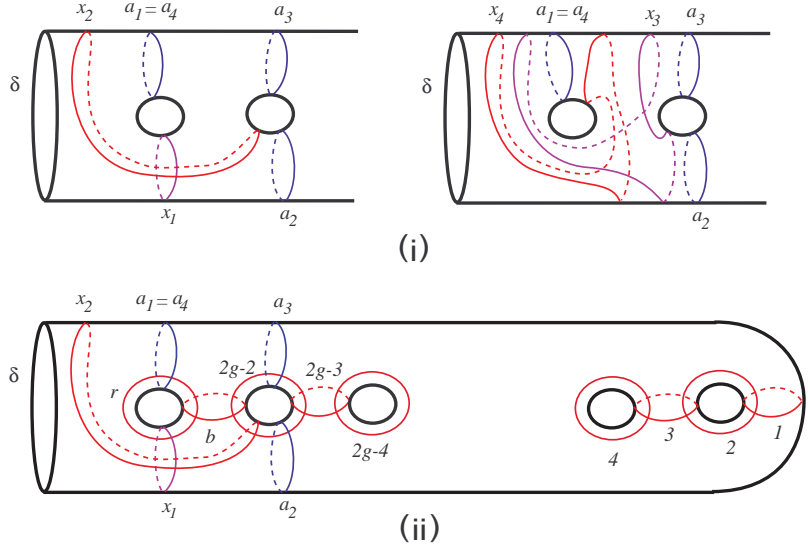}
      \caption{The curve labeled $i$ is $c_i$ for $i=1,2,\ldots, 2g-2$.}
      \label{lantern-59}
   \end{center}
 \end{figure}

Consider the four-holed disk $D$ in Figure~\ref{lantern-57}(ii) embedded in the surface $\Sigma$ of
genus $g\geq 2$ as in Figure~\ref{lantern-59}(i), so that $a_1=a_4$.

Suppose first that $n=2k$. It follows from Lemma~\ref{lem:comm} that $t_{x_3}^k t_{a_4}^{-k}  t_{x_4}^k  t_{a_3}^{-k}$
is a commutator and so is $t_{x_1} t_{a_2}^{-1} t_{x_2}  t_{a_1}^{-1}$. Since a conjugate of a commutator
is again a commutator, we have expressed $t_\delta^{n}$ as a product of $k+1$ commutators, proving the
theorem for even $n$.

Suppose now that $n=2k+1$. We rewrite the lantern relation~\eqref{eqn:lantern2} as
\begin{eqnarray}\label{eqn:lantern222}
t_{x_3}^{-1} t_\delta  =  t_{x_4} t_{a_3}^{-1}t_{a_4}^{-1}  t_{y_2}  t_{y_1}^{-1} .
\end{eqnarray}
From~\eqref{eqn:delta2nn} and~\eqref{eqn:lantern222}, we get
 \begin{eqnarray*}
t_{x_3}^{-1} t_\delta^{2k+1}
 &=&\left( \prod_{i=1}^{k}  (t_{x_1}t_{a_2}^{-1} t_{x_2} t_{a_1}^{-1})^{t_{x_3}^{i-1}} \right)
 (t_{x_3}^k  t_{a_4}^{-k-1} t_{x_4}^{k+1} t_{a_3}^{-k-1}) (t_{y_2}  t_{y_1}^{-1}).
\end{eqnarray*}
or
 \begin{eqnarray}\label{eqn:delta2k+1}
t_\delta^{2k+1}
 &=&\left( \prod_{i=1}^{k}  (t_{x_1}t_{a_2}^{-1} t_{x_2} t_{a_1}^{-1})^{t_{x_3}^{i}} \right)
 (t_{x_3}^{k+1}t_{a_4}^{-k-1}  t_{x_4}^{k+1} t_{a_3}^{-k-1} ) (t_{y_2}  t_{y_1}^{-1}).
\end{eqnarray}
It follows again from Lemma~\ref{lem:comm} that each one of the elements
$t_{x_1}t_{a_2}^{-1} t_{x_2} t_{a_1}^{-1}$ and $t_{x_3}^{k+1}t_{a_4}^{-k-1}  t_{x_4}^{k+1} t_{a_3}^{-k-1}$
is a commutator. Since $t_{y_2}  t_{y_1}^{-1}$ is also a commutator, it follows that
$t^{2k+1}_\delta$ is a product of $k+2$ commutators.

This proves the theorem.
\end{proof}

\begin{corollary} \label{cor:ez-d^2n}
Let $g\geq 2$, $h\geq 1$, and let $\Sigma$ be a compact connected oriented surface of genus $g$
with boundary. Let $\delta$ be one of the boundary components of $\Sigma$. Then for each $2-2h\leq k\leq 2h-2$,
$t_\delta^{k}$ is a product of $h$ commutators.
\end{corollary}


\subsection{Stable commutator lengths in mapping class groups of surfaces with boundary} \label{sclresults} \

Let $G$ be a group and let $[G,G]$ denote the commutator subgroup of $G$, the subgroup generated by the set of
all commutators $[x,y]=xyx^{-1}y^{-1}$ for $x,y\in G$. Let $x$ be an element in $[G,G]$, so that it can be
written as a product of commutators. The \emph{commutator length} $\cl (x)$ of $x$ in $G$ is defined as the
minimal number of commutators need to express $x$ as a product of commutators. The \emph{stable commutator length} of $x$ is the limit
\[ \displaystyle \scl (x)=\lim_{n\to\infty} \frac{\cl (x^n)}{n}.\]
This limit exists since the function $\cl$ is subadditive:
$\cl (x^{n+m})\leq \cl (x^{n}) + \cl (x^{m})$. Note also that $\cl (yxy^{-1})= \cl (x)$, $\cl (x^{-1})= \cl (x)$ and $\scl (x^k)=|k|\, \scl (x)$ for any $x\in [G,G]$ and $y\in G$. For more on the functions $\cl$
 and $\scl$, we refer to~\cite{Calegari}.

\begin{theorem} \label{thm:scl}
Let $g\geq 2$ and let $\Sigma$ be a compact connected oriented surface of genus $g$ with at least
one boundary components. Let $\delta$  be one of the boundary components of $\Sigma$. Then for any nonzero integer $n$
the commutator length of $t_\delta^n$ is
\[ \cl (t_\delta^n)=\lfloor (|n|+3)/2 \rfloor.
\]
In particular, we have $\scl (t_\delta)=1/2$.
\end{theorem}
\begin{proof} Since, $\cl (t_\delta^{-n})=\cl (t_\delta^{n})$ we may assume that $n$ is positive.
By Theorem~\ref{thm:ez-d^2n},  $\cl (t_\delta^{n})\leq \lfloor (n+3)/2 \rfloor $.

Suppose, on the other hand, that $c(t_\delta^{n})=h$. Writing $t_\delta^{n}$ as
a product of $h$ commutators gives us a genus--$g$ surface bundle over
a surface of genus $h$ with a section of self-intersection $-n$. By Proposition~\ref{adjunctionbounds},
we get $n\leq 2h-2$. Hence, we have $c(t_\delta^{n}) \geq (n+2)/2$.
The theorem follows from this.
\end{proof}

Theorem~\ref{thm:scl} computes the first precise nonzero value of the stable commutator
length of an element in the mapping class group of a surface. In earlier works bounds
on stable commutator lengths of elements of mapping class groups were established without
such an explicit calculation \cite{EK, Kork}.

\begin{remark} (An alternative proof of ${\rm scl}(t_\delta) \leq \frac{1}{2}$.) \label{sclalternative}
The inequality ${\rm scl}(t_\delta) \leq \frac{1}{2}$ may alternatively be obtained by using quasimorphisms: Given a group $G$, a function $\phi:G\rightarrow \mathbb{R}$ is called a \emph{quasimorphism} if there exists a least constant
$D_\phi\geq 0$, called the \emph{defect} of $\phi$,
such that $|\phi(xy)-\phi(x)-\phi(y)|\leq D_\phi$ for all $x,y\in G$.
A quasimorphism $\phi$ is called \emph{homogeneous} if it satisfies the additional property that
$\phi(x^{n})=n\phi({x})$
for all $x\in G$ and integer $n$. For more on quasimorphisms, we refer to~\cite{Calegari}.

We recall the basic properties of homogeneous quasimorphisms we use. If $\phi$ is a homogeneous quasimorphism,
then $\phi$ is constant on conjugacy classes. Therefore, a homogeneous quasimorphism on the mapping class group
 takes the same value on Dehn twists about nonseparating simple closed curves. For any two elements $x,y$ in $G$, if $xy=yx$, then
\begin{eqnarray*}
 |\phi(xy)-\phi(x)-\phi(y)|
   &  =  & \lim_{n\rightarrow\infty}|\phi(xy)-\phi(x)-\phi(y)| \\
   &  =  & \lim_{n\rightarrow \infty}\frac{1}{n}|\phi((xy)^n)-\phi(x^n)-\phi(y^n)|\\
   &  =  & \lim_{n\rightarrow \infty}\frac{1}{n}|\phi(x^ny^n)-\phi(x^n)-\phi(y^n)| \\
   &\leq & \lim_{n\rightarrow \infty}\frac{1}{n}D_\phi\\
   &  =  &0.
 \end{eqnarray*}
Hence $\phi(xy)=\phi(x)+\phi(y)$.

Let $Q$ be the set of homogeneous quasimorphisms on $G$. Bavard's main theorem in~\cite{Ba} states that for any
$x\in [G,G]$,
 \begin{eqnarray} \label{eqn:Bavard}
  {\rm scl}(x)=\displaystyle\sup_{\phi\in Q}\frac{|\phi(x)|}{2\, D_\phi}.
\end{eqnarray}

Now, the proof proceeds as follows: Let $\phi$ be a homogeneous quasimorphism on the mapping class group of
$\Sigma$, which is a surface of genus $g\geq 2$. Let $\delta$ be a boundary component.
We may choose six nonseparating curves $a_1,a_2,a_3,x_1,x_2,x_3$ on $\Sigma$ such that we have the
lantern relation $t_\delta t_{a_1} t_{a_2} t_{a_3}=t_{x_1} t_{x_2}t_{x_3}$. We then have
  \begin{eqnarray*}
   \phi(t_{x_1} t_{x_2})&=&\phi(t_\delta t_{a_1} t_{a_2} t_{a_3} t_{x_3}^{-1})\\
    &=&\phi(t_\delta)+\phi(t_{a_1})+\phi(t_{a_2})+\phi(t_{a_3})-\phi(t_{x_3})\\
    &=&\phi(t_\delta)+2\phi(t_{a_1}).
  \end{eqnarray*}
 From the definition of quasimorphism and the properties of $\phi$
  \begin{eqnarray*}
   D_\phi&\geq&|\phi(t_{x_1} t_{x_2})-\phi(t_{x_1})-\phi(t_{x_2})|\\
    &=&|\phi(t_\delta)+2\phi(t_{a_1})-\phi(t_{x_1})-\phi(t_{x_2})|\\
    &=&|\phi(t_\delta)|.
  \end{eqnarray*}

Thus, by Bavard's result~(\ref{eqn:Bavard}), we have ${\rm scl}(t_\delta)\leq \frac{1}{2}$.
\end{remark}

Lastly, we look at the relevant \emph{stable} mapping class group. Let $\Sigma_g^{k}$ denote the
surface of genus $g$ with $k\geq 2$ boundary components.  Let us distinguish a boundary component $\partial$. By attaching a torus $T$ with two boundary components to $\Sigma_g^{k}$ along $\partial$, we obtain a surface $\Sigma_{g+1}^k$ of genus $g+1$ with a distinguished boundary component, $\partial$, the common boundary component of $T$ and $\Sigma_{g+1}^k$, and an injection $\Gamma_g^k\to \Gamma_{g+1}^k$. By taking the direct limit
we get a group
\[ \Gamma_\infty^{k-1} = \lim _{g\to\infty} \Gamma_g^k.\]
The group $\Gamma_\infty^{k-1}$ is the group of isotopy classes of compactly supported diffeomorphisms
of a one-ended countable infinite genus surface with $k-1$ boundary components.

If $\delta$ is a boundary component of $\Sigma_{g}^k$ other than $\partial$, then the Dehn twist $t_\delta$ may be seen as an element of $\Gamma_\infty^{k-1}$. Theorem~\ref{thm:scl} implies that the stable commutator length of $t_\delta$ in $\Gamma_\infty^{k-1}$ is $\frac{1}{2}$. The same conclusion holds if one stabilizes the surface along possible different boundary components but not along $\delta$.
(In this case the stable mapping class group \emph{depends} on the stabilization.)
In particular, $\scl (t_\delta)$ is stable. This may be compared with Harer's homology stability \cite{Harer}, and contrasts a result of Kotschick~\cite{kot} who showed that the function $\scl$ vanishes on $\Gamma_\infty=\Gamma_\infty^0$.

\subsection{Surface bundles with maximal sections} \

We note that in any group $G$, if $a_i$ commutes with $b_j$ then $[a_1,a_2][b_1,b_2]= [a_1b_1,a_2b_2]$. This fact will be used repeatedly below.

\begin{theorem} \label{maintheorem1}
Let $h\geq 1$. For every $g\geq 2$ and $2-2h\leq k \leq 2h-2$, there exists a $\Sigma_g$--bundle over $\Sigma_h$, admitting a section of self-intersection $k$, and in particular, a maximal section. Moreover, provided $g\geq 8h-8$, there exists a $\Sigma_g$--bundle over $\Sigma_h$, admitting $4h-3$ disjoint sections $S_i$ with $[S_i]^2= i$, for $2-2h\leq i \leq 2h-2$.
\end{theorem}

\begin{proof}
Let $\Sigma$ be a compact connected oriented surface of genus $g$ with one boundary component $\delta$. By Theorem~\ref{thm:ez-d^2n},
we may express $t_\delta^{2-2h}$ as a product of $h$ commutators in $\Gamma_{g}^1$. Capping off the boundary component $\delta$
of $\Sigma$ gives us a relation in $\Gamma_{g}$ which is the monodromy representation of a genus--$g$ surface bundle
over a genus--$h$ surface with a section, and the self-intersection number of this section is $2h-2$. This proves the first assertion of the theorem.

For the second part of the theorem, let $Z$ be a compact connected oriented surface of genus $l \geq 0$ with
$4h-3$ boundary components, $d_i$,  $2-2h\leq i \leq 2h-2$. For each $2-2h\leq i \leq 2h-2$ with $i\neq 0$, let
$Z_i$ be a compact connected orientable surface of genus two with two boundary components $d'_i$ and $\delta_i$.
Glue each $Z_i$ to $Z$ by identifying $d_i$ and $d'_i$ to obtain a connected oriented surface $\Sigma$ of genus
$g=8h-8+ l $ with $4h-3$ boundary components $\delta_i$, where $\delta_0=d_0$.

By Corollary~\ref{cor:ez-d^2n}, there are diffeomorphisms $f_{i1},g_{i1},f_{i2},g_{i2},\ldots,f_{ih},g_{ih}$
of $\Sigma$ supported on the subsurface  $Z_i$ such that
\[t_{\delta_i}^i=[f_{i1},g_{i1}][f_{i2},g_{i2}]\cdots [f_{ih},g_{ih}].  \]
Note that if $i\neq j$, then $f_{is}$ and $g_{is}$ commute with $f_{jt}$ and $g_{jt}$. It now follows  that
\[ \displaystyle \prod_{i=2-2h}^{2h-2} t_{\delta_i}^i
\]
is a product of $h$ commutators. Hence, we obtained a relation yielding a monodromy representation for a surface bundle with the desired properties.
\end{proof}

\begin{remark}
The second construction in our proof of Theorem~\ref{maintheorem1} can be derived from the first one geometrically. To do so, we first take separate $\Sigma_{g_i}$--bundle $f_i$ over $\Sigma_h$ with a section $S_i$ of self-intersection $[S_i]^2=i$ for each integer $i$ in the interval $[2-2h, 2h-2]$, for any chosen collection of $g_i \geq 2$ satisfying $\sum_i g_i = g \geq 8h -8$. Then we observe that the relations handed us by Corollary~\ref{cor:ez-d^2n} hold on surfaces with more than one boundary component as well. Thus we can get self-intersection zero sections of each $f_i$ disjoint from $S_i$. Hence, we can \textit{section sum} the $f_i$ along these self-intersection zero sections to obtain a $\Sigma_g$--bundle over $\Sigma_h$, provided $g \geq 8h-8$.
\end{remark}

We now look at the remaining cases for the fiber and the base genera not covered by Theorem~\ref{maintheorem1}.

If $g=0$, $h \geq 0$, then the total space of the bundle is an $S^2$--bundle over $\Sigma_h$, which is classified by the Euler number $k$ and the base genus $h$. Each bundle corresponding to Euler number $k$ admits a section of self-intersection $k$. When $g=1$, we swing to another extreme situation: By Proposition \ref{propg=1}, a section of a $T^2$--bundle over any $\Sigma_h$ can only attain zero self-intersection number.

When the base genus $h=0$, there is a unique $\Sigma_g$--bundle over $S^2$ for $g > 1$, namely the trivial bundle $\Sigma_g \x S^2$ given by the projection onto the second component. There are many $T^2$--bundles over $S^2$, however only one of them admits a section; the trivial bundle on $T^2 \x S^2$. (See for instance \cite{BK} (Lemma 10).) So the self-intersection number of a section of a $\Sigma_g$--bundle over $S^2$ for $g \geq 1$ is always zero.

These observations in the remaining cases can be summarized as:

\begin{proposition} \label{lowgeneracases} For any $k \in \Z$, and $h \geq 0$, there exists an $S^2$--bundle over $\Sigma_h$ admitting a section $S$ with $[S]^2 =k$. For any $g \geq 1$, the only $\Sigma_g$--bundle over $S^2$ admitting a section $S$ is the trivial one, for which $[S]^2$ is necessarily zero. For any $h \geq 0$, any section $S$ of a $T^2$--bundle over $\Sigma_h$ has $[S]^2=0$.
\end{proposition}

\vspace{0.3cm}
\section{\mbox{Lefschetz fibrations with arbitrarily large number of critical points}}\label{lfresults}

In this section, we focus on the following problem: \textit{For fixed non-negative integers $g$ and $h$,
is there an upper bound on the number of critical points of a relatively minimal genus--$g$ Lefschetz fibration over a genus--$h$ surface, admitting a maximal section?}

One can easily inflate the number of critical points of Lefschetz fibrations with fixed fiber and base genera by introducing reducible fibers or by performing fiber sums with the same genus fibrations over the $2$-sphere. Our assumptions on relative minimality of the fibrations and on the existence of maximal sections (see Proposition \ref{fibersumindecomposablity}) are therefore essential. On the other hand, it was shown by Ivan Smith~\cite{Smith} and independently by Andras Stipsicz~\cite{Stipsicz} that the self-intersection of a section of a non-trivial Lefschetz fibration over the $2$-sphere can be at most $-1$, which is achieved by many Lefschetz fibrations, and most importantly by those that are obtained from Lefschetz pencils. So the case $h=0$ of our question is the one in \cite{Au}.

The problem is vacuous for $g=0$ because of our natural assumption on the minimal relativity. For $g=1$, and for any fixed $h \geq 0$, Proposition \ref{propg=1} implies that the number of critical points $m$ of a Lefschetz fibration is determined by the self-intersection number $n$ of a section by the equaiton $m= - 12 n$. In particular, every section should have the same self-intersection number. So the number of critical points is bounded in this case. As shown by Ivan Smith \cite{Smith2} (also see~\cite{Sato}), the number of critical points is also bounded when $h=0$ and $g=2$. Our main theorem in this section shows that this fails to be true when the base genus $h$ is positive:

We will need the following well-known lemma in the proof of our main theorem of this section:

\begin{lemma} \label{lem:chain}
Let $c_1,c_2,\ldots, c_{2l+1}$ be a sequence of simple closed curves on an oriented surface
such that $c_i$ and  $c_j$ are disjoint if $|i-j|\geq 2$ and that $c_i$ and $c_{i+1}$ intersect at one point.
A regular neighborhood of $c_1\cup c_2\cup \cdots \cup  c_{2l+1}$ is a subsurface of genus $l$ with
two boundary components, say $d_1$ and $d_2$. We then have the chain relation
\[
\left( t_{c_1}t_{c_2}\cdots t_{c_{2l+1}} \right)^{2l+2}=t_{d_1}t_{d_2}.
\]
\end{lemma}

\begin{theorem} \label{maintheorem2}
Let $g\geq 2$ and $h\geq 1$ be fixed integers. For any positive integer $M$, there exists a relatively minimal genus--$g$ Lefschetz fibration over a surface of genus $h$ admitting a section of maximal self-intersection such that the number of critical points is greater than $M$.
\end{theorem}

\begin{proof} By Proposition~\ref{adjunctionbounds},
the self-intersection of a maximal section of a genus--$g$ Lefschetz fibration over a surface
of genus $h\geq 1$ is $2h-2$. Therefore, it suffices to write $t_\delta^{2-2h}$ as a product of $h$ commutators and
a number of right Dehn twists such that the number of right Dehn twists is at least $M$.

Suppose first that $h\geq 2$. Consider the surface $\Sigma_{g}^1$ in Figure~\ref{lantern-59}(ii) and the curves on it.
By~\eqref{eqn:delta2nn}, we have the following equality the mapping class group of $\Sigma_{g}^1$;
\begin{eqnarray} \label{eqn:del-2n}
t_\delta^{-2k}
 &=&( t_{a_3}^k t_{x_4}^{-k}  t_{a_4}^k  t_{x_3}^{-k} )
 \left( \prod_{i=1}^{k}  (t_{a_1} t_{x_2}^{-1} t_{a_2} t_{x_1}^{-1})^{t_{x_3}^{k-i}} \right).
\end{eqnarray}
By writing $k=h-1$, $C_{h-1}=t_{a_3}^{h-1} t_{x_4}^{1-h}  t_{a_4}^{h-1}    t_{x_3}^{1-h}$,
$C= t_{a_1} t_{x_2}^{-1} t_{a_2} t_{x_1}^{-1}$ and $C_i= t_{x_3}^{i} C t_{x_3}^{-i}$, $1\leq i\leq h-2$,
we may rewrite \eqref{eqn:del-2n} as
 \begin{eqnarray}\label{eqn:delta2-2h}
 t_\delta^{2-2h} &=& C_{h-1} C_{h-2}\cdots  C_2 C_1 C,
\end{eqnarray}
where each $C_i$ is a commutator.

On the other hand, by Lemma~\ref{lem:chain}, we have the relation $ \left( t_{a_1}t_{r}t_{b}\right)^4= t_{x_2}t_{a_3}$. From this one we easily get
\begin{eqnarray}\label{eqn:T-1x}
t_{r} t_{a_1} t_{b} t_{r} \left( t_{a_1}t_{r}t_{b}\right)^2= t_{x_2}t_{a_3} t_{a_1}^{-2}.
\end{eqnarray}
Let $T_1$ denote the left hand side of the equality~\eqref{eqn:T-1x}, so that $T_1$ is a
product of $10$ right Dehn twists. Thus, for any integer $m$, we have
\begin{eqnarray*}
T_1^m = t_{x_2}^{m} t_{a_3}^{m} t_{a_1}^{-2m}
\end{eqnarray*}
or
\begin{eqnarray}\label{eqn:T-1x=1}
T_1^m  t_{a_1}^{2m} t_{a_3}^{-m}  t_{x_2}^{-m} =1.
\end{eqnarray}

By Lemma~\ref{lem:chain}, we also have the relation
\begin{eqnarray}\label{eqn:T-2x}
 \left( t_{c_1} t_{c_2}\cdots t_{c_{2g-3}} t_{c_{2g-2}} t_{b} \right)^{2g}= t_{x_1}t_{a_1}.
\end{eqnarray}
The left hand side of this equality may be written as $T_2\left( t_{c_1} t_{c_2}\cdots t_{c_{2g-3}} \right)^{2g-2} $,
where $T_2$ is a product of right Dehn twists. (The number of Dehn twists in $T_2$ is $8g-6$.)
Applying Lemma~\ref{lem:chain} once again,
we see that it is also equal to $T_2 t_{a_2}t_{a_3}$. Hence, we obtain from (\ref{eqn:T-2x}) that
\begin{eqnarray*}
T_2^m =  t_{x_1}^{m} t_{a_1}^{m}t_{a_2}^{-m} t_{a_3}^{-m},
\end{eqnarray*}
or
\begin{eqnarray}\label{eqn:T-2x=1}
t_{x_1}^{-m}  t_{a_1}^{-m}  t_{a_2}^{m}  t_{a_3}^{m}  T_2^m =1.
\end{eqnarray}

Finally, by using (\ref{eqn:T-1x=1}) and (\ref{eqn:T-2x=1}), we write $C$ as
\begin{eqnarray*}
 C &=& t_{a_1} t_{x_2}^{-1} t_{a_2} t_{x_1}^{-1}\\
 &=& t_{x_2}^{-1} t_{x_1}^{-1}  t_{a_1} t_{a_2} \\
 &=&  (  T_1^m  t_{a_1}^{2m} t_{a_3}^{-m}  t_{x_2}^{-m}  ) t_{x_2}^{-1} t_{x_1}^{-1} t_{a_1} t_{a_2}
 (  t_{x_1}^{-m}  t_{a_1}^{-m}  t_{a_2}^{m}  t_{a_3}^{m}  T_2^m ) \\
  &=&  T_1^m  t_{x_2}^{-m-1}  t_{x_1}^{-m-1} t_{a_1}^{m+1} t_{a_2}^{m+1}  T_2^m\\
  &=&  \left( T_1^m  t_{x_2}^{-m-1}  t_{x_1}^{-m-1} t_{a_1}^{m+1} t_{a_2}^{m+1} T_1^{-m}\right) T_1^m  T_2^m.
 \end{eqnarray*}
If we let $C_0^{(m)}=T_1^m  (t_{x_1}^{-m-1}  t_{x_2}^{-m-1} t_{a_1}^{m+1} t_{a_2}^{m+1})  T_1^{-m}$, then $C_0^{(m)}$ is a commutator.
By~\eqref{eqn:delta2-2h}, we then have
 \begin{eqnarray*}
 t_\delta^{2-2h} &=& C_{h-1} C_{h-2}\cdots  C_2 C_1 C_0^{(m)} T_1^mT_2^m,
\end{eqnarray*}
so that $ t_\delta^{2-2h}$ is expressed as a product of $h$ commutators and $4(2g+1)m$ right Dehn twists. Taking $m$ large enough finishes the proof of the theorem in the case $h\geq 2$.

Suppose now that $h=1$, so that $2h-2=0$.
Consider the surface $\Sigma_{g}^1$ and the curves on it given in Figure~\ref{lantern-59}(ii) once again. Let $2\leq l\leq g$ and let $d_1$ and $d_2$ be the boundary components of a regular neighborhood of $c_{1}\cup c_{2}\cup c_{3}\cup \cdots \cup c_{2l-1}$, where $c_{2g-1}=b$ and  $c_{2g}=r$.

By Lemma~\ref{lem:chain}, we have the relation $ \left( t_{c_1}t_{c_2}\cdots t_{c_{2l-1}} \right)^{2l}=t_{d_l}t_{d_2}$.
From this we get
\begin{eqnarray}\label{eqn:T-l}
\left( t_{c_2}\cdots t_{c_{2l-1}} \right) \left( t_{c_1}t_{c_2}\cdots t_{c_{2l-1}} \right)^{2l-2} \left( t_{c_1}t_{c_2}\cdots t_{c_{2l-2}} \right) =t_{d_l} t_{d_2} t_{c_1}^{-1} t_{c_{2l-1}}^{-1}.
\end{eqnarray}
Let $T(l)$ denote the left hand side of (\ref{eqn:T-l}), so that $T(l)$ is a product of right Dehn twists.
Thus, for any integer $m$ we have
\begin{eqnarray*}
(T(l))^m = t_{d_l}^{m} t_{d_2}^{m} t_{c_1}^{-m}t_{c_{2l-1}}^{-m},
\end{eqnarray*}
or
\begin{eqnarray}\label{eqn:T-l=1}
t_\delta^{0} =1 =t_{d_l}^{-m} t_{d_2}^{-m} t_{c_1}^{m}t_{c_{2l-1}}^{m} (T(l))^m.
\end{eqnarray}
If we let $C(l) =t_{d_l}^{-m} t_{d_2}^{-m} t_{c_1}^{m}t_{c_{2l-1}}^{m}$, then $C(l)$ is a commutator.
By (\ref{eqn:T-l=1}), we then have
\begin{eqnarray}\label{eqn:T-l=cl}
 t_\delta^{0} =1 =C(l)(T(l))^m.
\end{eqnarray}

Again, taking $m$ large enough finishes the proof in the case $h= 1$.
\end{proof}

\begin{remark}
It should be evident that the techniques used in the proof of Theorem~\ref{maintheorem2} fall short to cover the remaining case $h=0$, and thus, the question remains to be open for $h=0$ and $g \geq 3$.
\end{remark}

\begin{remark} Proposition \ref{fibersumindecomposablity} implies that the Lefschetz fibration and section pairs we construct in Theorem \ref{maintheorem2} are fiber sum indecomposable. It is apriori unclear whether or not these fibrations (but not the pairs) can be decomposed as fiber sums where one of the fibrations is over the $2$-sphere. We shall note, however, by employing slightly different manipulations of mapping class group relations, we can construct examples as in Theorem \ref{maintheorem2} meeting these additional conditions as well, relying on Smith's fillability criterion in \cite{Smith}. Nevertheless, we will skip this rather repetitive inclusion here, which does not seem to relate to the core problem in question.
\end{remark}

\begin{rmk}
Theorem \ref{maintheorem2} shows the existence of genus--$g$ Lefschetz fibrations with maximal sections over surfaces of positive genera, similar to our result casted in Theorem \ref{maintheorem1}.
\end{rmk}

\vspace{0.2in}
\noindent \textit{Acknowledgements.} The first author was partially supported by the NSF grant DMS-0906912. The second author thanks Max-Planck Institut f\"ur Mathematik in Bonn for its generous support and wonderful research environment. The third author thanks Hisaaki Endo, Kenta Hayano, Dieter Kotschick, and Masatoshi Sato for their comments on the content of this paper.

\end{document}